\documentclass[10pt]{amsart}  
\usepackage[top=1.5in, bottom=1.5in, left=1.25in, right=1.25in]	{geometry}
\usepackage{cite}

\usepackage{mathtools}
\mathtoolsset{showonlyrefs,showmanualtags} 

\usepackage{parskip}
\usepackage{amsmath}
\usepackage[initials]{amsrefs}
\usepackage{amsfonts}
\usepackage{color}
\usepackage{amssymb,bbm}
\usepackage{eucal}
\usepackage{framed}
\usepackage{multicol}
\usepackage{graphicx}
\usepackage[makeroom]{cancel}
\usepackage{fancyhdr}
\usepackage{enumitem}
\usepackage{tikz}
\usepackage{hyperref}
\hypersetup{
  colorlinks   = true, 
  urlcolor     = blue, 
  linkcolor    = blue, 
  citecolor   = blue 
}
\usetikzlibrary{arrows}
\DeclareFontFamily{U}{wncy}{}
    \DeclareFontShape{U}{wncy}{m}{n}{<->wncyr10}{}
    \DeclareSymbolFont{mcy}{U}{wncy}{m}{n}
    \DeclareMathSymbol{\Sh}{\mathord}{mcy}{"58} 

\allowdisplaybreaks

\theoremstyle{plain} 
\newtheorem{lemma}[equation]{Lemma} 
 
\newtheorem{theorem}[equation]{Theorem}

\theoremstyle{definition}

\newtheorem{defi}{Definition}[section]

\theoremstyle{remark}

\numberwithin{equation}{section}

\newcommand{\set}[2]{\left\{\,#1 : #2 \, \right\}}
\newcommand{\inn}[2]{\left\langle #1 , #2 \right\rangle}
\newcommand{\ave}[1]{\left\langle #1 \right\rangle}

\begin{document}

\title[Sparse T1]{The Sparse T1 Theorem}

\author[Lacey]{Michael T. Lacey}   

\address{ School of Mathematics, Georgia Institute of Technology, Atlanta GA 30332, USA}
\email {lacey@math.gatech.edu}
\thanks{Research supported in part by grant  NSF-DMS-1600693}

\author[Mena]{Dar\'io Mena Arias}   

\address{ School of Mathematics, Georgia Institute of Technology, Atlanta GA 30332, USA}
\email {dario.mena@math.gatech.edu}

\begin{abstract}
We impose  standard $ T1$-type assumptions on a Calder\'on-Zygmund operator $ T$, and deduce that
for bounded compactly supported functions $ f, g$ there is a sparse bilinear form $ \Lambda $ so
that 
\begin{equation*}
\lvert  \langle T f, g \rangle\rvert \lesssim \Lambda (f,g). 
\end{equation*}
The proof is short and elementary.  The sparse bound quickly implies all the standard mapping
properties of a Calder\'on-Zygmund on a (weighted) $ L ^{p}$ space.  
\end{abstract}

\maketitle

\section{Introduction} 

We recast the statement of the $ T1 $ theorem of David and Journ\'e \cite{MR763911}, replacing the
conclusion that the operator $ T$ admits a quantitative bound on its $ L ^2 $-norm, with the
conclusion that $ T$ admits a quantitative sparse bound. 
From the sparse bound, one can quickly derive a wide range of (weighted) $ L ^{p}$ type inequalities
for $ T$. 
That is, the theory devoted to deriving these properties for $ T$ can be replaced by the much
simpler approach via sparse operators.  

We say that an operator $ T$ is a  Calder\'on-Zygmund operator on $ \mathbb R ^{d}$ if 
(a) it is bounded on $ L ^{2}$,  (b) there is a kernel $ K (x,y) \;:\; \mathbb R ^{d} \times \mathbb
R ^{d} \setminus \{ (x,x) \;:\; x\in \mathbb R ^{d}\} \to \mathbb R $ so that for functions
functions $ f, g$ smooth,  compactly supported, have disjoint closed supports, 
\begin{equation*}
B_T (f,g) = 
\langle T f, g \rangle = \int\!\!\int K (x,y) f (y) g (x)\; dy dy . 
\end{equation*}
(c) For some constant $\mathcal K_T$, the kernel $ K (x,y)$ satisfies 
\begin{gather}\label{e:size}
 \lvert  K (x,y)\rvert \leq \frac{\mathcal K_T}{ \lvert  x-y\rvert }, \qquad x\neq y \in \mathbb R ^{d}, 
 \\
 \lvert  K (x,y) - K (x',y)\rvert < \mathcal K_T \frac{ \lvert  x-x'\rvert ^{\eta } } { \lvert  x-y\rvert
^{d +
\eta } }, \qquad  0 < 2 \lvert  x-x'\rvert < \lvert  x-y\rvert   . 
\end{gather}
And, the same condition with the roles of $ x$ and $ y$ reversed.  
Above, $ \eta >0$ is a fixed small constant.

A sparse bilinear form $ \Lambda (f,g)$ is defined this way:  There is a collection of cubes $
\mathcal S$, 
so that for each $ S\in \mathcal S$, there is an $ E_S \subset S$ so that (a) $ \lvert  E_S\rvert >
c \lvert  S\rvert  $, 
and (b)  $\lVert \sum_{S\in \mathcal S} \mathbf 1_{E_S}\rVert _ \infty \leq c ^{-1}$. 
Then, set 
\begin{equation*}
\Lambda (f,g) = \sum_{S\in \mathcal S} \langle f \rangle_S \langle g \rangle_S \lvert  S\rvert,  
\end{equation*}
where $ \langle f \rangle_S = \lvert  S\rvert ^{-1} \int _{S} f (x) \; dx $.  
Here, we will not focus on the role of the constant $ 0 < c < 1$, and remark that many times it is
assumed that the sets $ E _S $ 
being pairwise disjoint, that is $\lVert \sum_{S\in \mathcal S} \mathbf 1_{E_S}\rVert _ \infty =1 $.
 
Our generalization does not affect the outlines of the theory, and makes some arguments somewhat
simpler.  

It is very useful to think of $ \Lambda (f,g)$ as a \emph{positive bilinear Calder\'on-Zygmund
form.}  
In particular, all the standard   inequalities can be quickly proved for $ \Lambda $. 
And, for weighted inequalities, it is easy to derive bounds that are sharp in the $ A_p$
characteristic.  

Our formulation of the $ T1$ theorem considers the usual $ L ^{1}$ testing condition on $ T$,
phrased in bilinear language. 

\begin{theorem}\label{t:T1}  Suppose that $ T$ is a Calder\'on-Zygmund operator on $ \mathbb R
^{d}$, and moreover there is a constant $ \mathbf T$ so that for all cubes $ Q$  and functions $
\lvert  \phi \rvert < \mathbf 1_{Q} $, there holds 
\begin{equation}\label{e:T1}
\lvert  B_T (\mathbf 1_{Q}, \phi )\rvert 
+\lvert  B_T (\phi ,\mathbf 1_{Q} )\rvert \leq \mathbf T \lvert  Q\rvert.  
\end{equation}
Then there is a constant $ C = C ( \mathcal K_T, \mathbf T, d , \eta )$ so that for all bounded
compactly supported
functions $ f, g$, there is a sparse operator $ \Lambda $ so that 
\begin{equation}\label{e:T<}
\lvert  B_T (f,g)\rvert < C \Lambda (\lvert  f\rvert, \lvert  g\rvert  ).   
\end{equation}
\end{theorem}

The proof is elementary, using (a) facts about averages and conditional expectations; (b)
random dyadic grids as a convenient tool to reduce the complexity of the argument; (c) 
orthogonality of martingale transforms, and the most sophisticated fact (d) a sparse bound for  a
certain bilinear square function, with complexity, detailed in Lemma~\ref{l:uv}.   In addition, the
testing condition \eqref{e:T1} appears solely in the construction of the stopping times.  
The proof is carried out in \S \ref{s:proof}. There are many terms, organized so that there is one
crucial term, in \S \ref{s:cruical}.  Almost all of the remaining cases use standard off-diagonal
considerations, and the simple argument to prove the sparse bound for a martingale transform.  This
is detailed in \S \ref{s:lemma}.   

The consequences of the sparse bound \eqref{e:T<} are: 
\begin{enumerate}
\item The weak type $ (1,1)$ inequality, and the $ L ^{p}$ inequalities, for $ 1< p < \infty $.
These hold with the sharp dependence upon $ p$.  To wit, using $ \lVert M \;:\; L ^{p} \mapsto L
^{p}\rVert \lesssim p' = \frac p {p-1}$, we have 
\begin{align*}
\Lambda (f,g) &= \int \sum_{S\in \mathcal S} \langle f \rangle_S \langle g \rangle_S \mathbf 1_{S}
\; dx 
 \lesssim \int  \sum_{S\in \mathcal S} \langle f \rangle_S \langle g \rangle_S \mathbf 1_{E_S} \;
dx 
\\
& \leq \int M f \cdot M g \; dx \leq \lVert M  f\rVert_p \lVert Mg\rVert_{p'} \lesssim p \cdot p'
\lVert f\rVert_p \lVert g\rVert_{p'}.
\end{align*}

\item The weighted version of the same, relative to $ A_p$ weights. The dependence upon the $ A_p$
characteristic is sharp, for $ 1< p < \infty $, and the best known for the case of $ p=1$.  See the
arguments in \cite{MR3085756}.  

\item  The  exponential integrability results of Karagulyan \cite{MR1913610,MR3124931}.
 
\end{enumerate}

Our statement of the $ T1$ theorem follows the `testing inequality' approach of the Sawyer two
weight theorems \cites{MR676801,MR930072}, 
and the statement in Stein's monograph \cite{MR1232192}.  Our approach is a descendant of the
radically dyadic approach of Figuel 
\cite{MR1110189}, further influenced by the martingale approach of Nazarov-Treil-Volberg
\cite{MR1998349}.
(Also see \cite{MR2956641}.)
Our use of the stopping cubes follows that of the proof of  the two weight Hilbert transform
estimate \cite{MR3285857}.

The bound by sparse operators has been an active and varied recent research topic.  It had a remarkable success in Lerner's approach to the $ A_2$ bound \cite{MR3085756},  which cleverly  bounded on the weighted norm  of a Calder\'on-Zygmund by a the norm of a sparse operator. 
The pointwise approach first established  in \cite{MR3521084}, with a somewhat different approach in \cite{2015arXiv150105818L}. 
The latter approach  has been studied  from several  different points of view \cites{2015arXiv151207247L,2015arXiv151000973B,2016arXiv160706319D,2016arXiv160603340V}. 
The form approach used here, is however successful in settings where the pointwise approach would fail, most notably the setting of the bilinear Hilbert transform \cite{2016arXiv160305317C},   Bochner Riesz multipliers \cite{2016arXiv160506401B}, and  oscillatory singular integrals \cite{2016arXiv160906364L}. 
The interested reader can consult the papers above for more information and references.  

This paper proves the sparse bound without appealing to any structural theory of Calder\'on-Zygmund operators such as boundedness of maximal truncations, which is the approach started in  \cite{2015arXiv150105818L}. 
The other prominent structural fact one could use is the Hyt\"onen structure theorem \cite{MR2912709}. 
This is the approach followed by  Culiuc-Di Plinio-Ou  \cite{161001958} also using bilinear forms.  They   show that this approach  has  further applications to the matricial setting, avoiding difficulties for the pointwise approach in this setting.

\section{Random Grids} 

All the proofs here will use Hyt\"onen's random dyadic grids from \cite{MR2912709}.  Recall again,
that the standard dyadic grid in $\mathbb R^{d}$ is 
$$
\mathcal D^{0} : = \bigcup_{k \in \mathbb Z} \mathcal D_k, \qquad  \mathcal D_k : = \set{2^k\left([0,1)^{d} + m\right)}{m  \in
\mathbb Z^d}.
$$
For a binary sequence $\omega : = (\omega_j)_{j \in \mathbb Z} \in \left( \{ 0, 1\}^{d} \right)^{\mathbb Z}$ we
define a general dyadic system by
$$
\mathcal D^{\omega} : = \set{Q \dotplus \omega}{ Q \in \mathcal D^{0} }, 
$$
where
$
Q \dotplus \omega = Q + \sum_{j : 2^{-j} < \ell Q} 2^{-j}\omega_j
$.
We consider the standard uniform probability measure on $\{0,1\}^{d}$, that is, it assigns $2^{-d}$
to every point.  We place on $\omega$, the probability measure $\mathbb P$, the corresponding
product measure on $\left( \{ 0, 1\}^{d} \right)^{\mathbb Z}$.  This way, we can see $(\mathcal D_{\omega})$ as a
collection of grids with a random set of parameters $\omega$.    For every $\omega$, these dyadic
grids satisfy the required properties, namely
\begin{enumerate}
\item For $P, Q \in \mathcal D^{\omega}$, $P \cap Q \in \{ P, Q, \emptyset \}$.
\item For fixed $k  \in \mathbb Z$, the collection $\mathcal D^{\omega}_{k} = \set{Q \in \mathcal D^{\omega}}{  \ell Q
= 2^{-k}}$ partitions $\mathbb R^d$.
\end{enumerate}
\begin{defi}[Good-bad intervals] Let $0 < \gamma < 1$ and a positive integer $r$ such that $r
\geq (1- \gamma)^{-1}$.  We say that $Q \in \mathcal D^{\omega}_{k}$ is $r$-bad, if there is
an integer $s \geq r$, and a choice of coordinate, so that the vectors 
$$
\omega_{k + \lfloor(1-\gamma)s\rfloor}, \omega_{k + \lfloor(1-\gamma)s\rfloor + 1}, \ldots ,
\omega_{k + s} \in \{ 0,1 \}^{d},
$$
all agree in that one coordinate.  If $Q$ is not $r$-bad, then it is called $r$-good.
\end{defi}
From now on, we are going to omit the dependence on $r$, and we will refer to the cubes as only good
or bad.   The following lemmas are well known.


\begin{lemma}\label{GeomGood}
If $Q$ is good, then for any cube $P$ with $2^{r} \ell Q < \ell P$ we have
$$
\textup{dist}(Q, \partial P) \gtrsim (\ell Q)^{\gamma} (\ell P)^{1- \gamma},
$$
where the implied constant is absolute.
\end{lemma}



\begin{lemma}\label{GoodIsLikely}
Fix $0 < \gamma < 1$ and $r > \gamma^{-1}$, then, there is a constant $C_d$ such that
$$
\mathbb P(\textup{$Q$ is good}) \geq 1 - C_d \gamma^{-1} 2^{-\gamma r}.
$$ 
\end{lemma}

For an arbitrary dyadic grid $\mathcal D^{\omega}$, every function $f \in L^{2}(\mathbb R^d)$ admits an
orthogonal decomposition
$$
f  = \sum_{Q \in \mathcal D^{\omega}} \Delta_Q f.
$$
Given a dyadic grid $\mathcal D^{\omega}$, we define the good and bad projections as
$$
P^{\textup{\tiny bad}}_{\omega} f : = \sum_{\substack{Q \in \mathcal D^{\omega} \\ \textup{$Q$ is bad}}}
\Delta_Q f, \qquad P^{\textup{\tiny good}}_{\omega} f : = \sum_{\substack{Q \in \mathcal D^{\omega} \\
\textup{$Q$ is good}}} \Delta_Q f.
$$
The following lemma says that in average, the bad projections tend to be small.

\begin{lemma}\label{BadProjSmall}
For all $1  < p < \infty$ there is an $\epsilon_p > 0$ such that for all $0 < \gamma < 1$ and $r >
\gamma^{-1}$ we have
$$
\mathbb E _{\omega} \| P^{\textup{\tiny bad}}_{\omega} f  \|_{L^p}^p \lesssim 2^{-\epsilon_p r} \| f
\|_{L^p}^p.
$$
\end{lemma}


Using this lemma, we can prove that it is enough to estimate bounds only for {\it good} functions,
in the following sense

\begin{lemma}\label{GoodIsEnough}
Let $1 <  p < \infty$.  If $T : L^p \mapsto  L^p$ is a bounded operator.  If $0 < \gamma < 1$ is fixed
and $r > C(1 + \log \frac{1}{\gamma})$, then
$$
\|  T :  L^p \mapsto L^p  \| \leq 4 M,
$$
where $M$ is the best constant in the inequality
$$
\mathbb E _{\omega}|\inn{T P^{\textup{\tiny good}}_{\omega} f}{ P^{\textup{\tiny good}}_{\omega} g}| \leq
M \| f \|_{L^p} \| g \|_{L^{p'}}.
$$
\end{lemma}

\section{The Proof of the Sparse Bound} \label{s:proof}

As a consequence of Lemma~\ref{GoodIsEnough}, it is enough for the remainder of the argument to
show this:
There is a choice of constant $ C>1$, so that for all $ f$ and $ g$ compactly supported, and almost
all 
grids $ \mathcal D^\omega  $, there is a sparse operator $ \Lambda =
\Lambda _{f,g, \mathcal D^\omega  , }$ so that 
\begin{equation}\label{e:TsparseGoal}
\lvert \langle T P^{\textup{\tiny good}} f, P^{\textup{\tiny good}} g \rangle\rvert 
\leq C \Lambda (|f|,|g|). 
\end{equation}
In view of the Lemma~\ref{l:sparse<}, the random sparse operator above can be replaced by a
deterministic one. 
Averaging over choices of grid will complete the proof.

Almost all random dyadic grids have the property that the functions $ f, g$ are supported on a
single good dyadic cube.  And, hence, on a sequence of dyadic cubes which exhaust $ \mathbb R
^{n}$. 
This  fact  and goodness are the only facts  about  random grids utilized, so we suppress the $
\omega $ dependence
below.   The inner product in \eqref{e:TsparseGoal} is  expanded 
\begin{equation}  \label{e:GG}
\langle T P^{\textup{\tiny good}} f, P^{\textup{\tiny good}} g \rangle 
= \sum_{\substack{P\in \mathcal D\\ \textup{$ P$ is good}}} 
\sum_{\substack{Q\in \mathcal D\\ \textup{$ Q$ is good}}} 
\langle T \Delta _{P} f , \Delta _{Q} g\rangle. 
\end{equation}
We will further only consider the case of $ \ell P \geq \ell Q$, the reverse case being addressed by
duality.  
The fact that $ P$ and $ Q$ are good will be suppressed, but always referenced when it is
used. 
And, by $ Q\Subset P$ we will mean that $ Q\subset P$ and $ 2 ^{r} \ell Q \leq \ell P$. Goodness of
$ Q$ then implies 
that 
\begin{equation}\label{e:good}
\textup{dist} (Q,   \textup{skel}P) \geq  (\ell Q) ^{\epsilon } (\ell P) ^{1- \epsilon }, 
\end{equation}
where $ \textup{skel} P$ is the union of $ \partial P'$, where $ P'$ is a child of $ P$.  
We will likewise suppress the role of the dyadic grid in our notation.  

As just mentioned, the two functions  $ f,g$ are supported on a single good cube $ P_0 \in  \mathcal
D$,
which we can take to be very large.   Therefore, we can
restrict the sum in \eqref{e:GG} to 
only cubes $ P, Q \subset P_0$.  The bound we
obtain will be independent of the choice of $ P_0$.  
The sum we consider is then broken into several subcases. 
\begin{align}
\sum_{\substack{P \;:\;  P\subset P_0}}  &
\sum_{\substack{Q \;:\;  Q\subset P_0 \\ \ell P \geq \ell Q}} \langle T \Delta _{P} f ,
\Delta _{Q} g\rangle
\\  \label{e:inside}
&= 
\sum_{\substack{P\;:\;   P\subset P_0}}  
\sum_{\substack{Q \;:\;   Q\Subset P}} \langle T \Delta _{P} f , \Delta _{Q} g\rangle 
&& \textup{(inside)} 
\\ \label{e:near} 
& \quad +\sum_{\substack{P\;:\;   P\subset P_0}} 
\sum_{\substack{Q\;:\; 2 ^{r} \ell Q \leq \ell P \\  Q\subset 3P \setminus P  }} \langle
T \Delta _{P} f , \Delta _{Q} g\rangle  && \textup{(near)} 
\\ \label{e:far}
& \quad +\sum_{\substack{P\;:\;   P\subset P_0}} 
\sum_{\substack{Q\;:\; \ell Q \leq \ell P \\  Q\cap 3P  = \emptyset   }} \langle T
\Delta _{P} f , \Delta _{Q} g\rangle  && \textup{(far)} 
\\ \label{e:neighbor} 
& \quad +\sum_{\substack{P\;:\;  P\subset P_0 }}  
\sum_{\substack{Q\;:\; \ell Q \leq \ell P \leq  2 ^{r } \ell Q \\ Q\cap 3P  \neq \emptyset
}} \langle T \Delta _{P} f , \Delta _{Q} g\rangle.  && \textup{(neighbors)} 
\end{align} 

\subsection{Stopping Cubes}
We define a sparse collection $\mathcal S$ of stopping cubes, and associated stopping values in the
following way:   Add $P_0$ to the
collection $\mathcal S$,  and set $ \sigma _{f} (P_0) = \langle \lvert  f\rvert  \rangle_{P_0} $, and
similarly for $ g$.  
In the recursive stage of the construction, for minimal $ S\in \mathcal S$,   define three sets 
\begin{itemize}
\item $F^1_S = \bigcup \set{S' \in \mathcal D(S)}{\ave{|f|}_{S'} > C_0  \sigma _{f} (S),\  S' \mbox{
maximal}}.$
\item $F^2_S = \bigcup \set{S' \in \mathcal D(S)}{\ave{|g|}_{S'} > C_0 \sigma _{g} (S),\ S' \mbox{
maximal}}.$
\item $F^3_S =   \bigcup \set{S' \in \mathcal D(S)}{  \langle \lvert  T \mathbf 1_{S}\rvert  \rangle_{S'} >
C_0 \mathbf T ,\  S' \mbox{ maximal}}.$
\end{itemize}
Let $F_S = F^1_S \cup F^2_S  \cup F ^{3}_S$, and $\mathcal F_S$ be the family of dyadic components
of $F_S$.   The weak-type bound for the dyadic maximal function  and the testing condition
\eqref{e:T1} implies that  there exists $C_0$
big enough, such that $|F_S| < \frac{1}{2}|S|$.  Recursively, add,
every $\mathcal F_S$ to the collection $\mathcal S$ to form a sparse collection.

 We set $ P ^{\sigma }$ to be the smallest stopping cube $ S$ that contains $ P$. And we set 
 $ Q ^{\tau }$ to be the smallest stopping cube $ S$ such that $ Q \Subset S$.  
  The Haar projection associated to $ S$ is  $ \Pi_ S g = \sum_{Q \;:\; Q ^{\tau } =S} \Delta _Q g$.

\subsection{The Inside Terms} \label{s:cruical}
We turn our attention to the main term, that of \eqref{e:inside}, for which there are three
subcases.  
The argument of $ T$ is $ \Delta _P f$, which we write as 
\begin{align} 
\Delta _P f & =\Delta _P f \mathbf 1_{P \setminus P_Q} +   \mathbf 1_{P_Q}  \Delta _P f 
\\ \label{e:insideSplit}
& = 
\Delta _P f \mathbf 1_{P \setminus P_Q} + \langle \Delta _P f \rangle_{P_Q} \cdot 
\begin{cases}
  \mathbf 1_{S} -    \mathbf 1_{S \setminus P_Q }   &  S = Q ^{\tau } \supset P_Q 
 \\
\mathbf 1_{S} +\mathbf 1_{P_Q \setminus S}   &  S = Q ^{\tau } \subsetneq  P_Q 
\end{cases}
\end{align}
where $ Q\Subset P$, and $ P_Q$ is the child of $ P$ that contains $ Q$.  

\subsubsection*{First Subcase}
Control the first term on the right in \eqref{e:insideSplit}  by  off-diagonal considerations.   Central
to all of these off-diagonal arguments are the class of forms $ B ^{u,v}$ defined in \eqref{e:uv},
which are in turn bounded by Lemma~\ref{l:uv}.

Since $ Q$ is a good cube,  the inequality \eqref{e:good} holds: That is $ Q$ is a long way from the
skeleton of $ P$. 
By \eqref{e:ODE},  we have 
\begin{align} 
\lvert \langle T (\Delta _P f \mathbf 1_{P \setminus P_Q}) , \Delta _Q g\rangle\rvert 
&\lesssim P _{\eta } ( \Delta _P f \mathbf 1_{P \setminus P_Q} ) (Q) \lVert \Delta _Q g\rVert_1 
\\ \label{e:PQ1}
&\lesssim [ \ell Q/ \ell P] ^{\eta ' }  \langle  \lvert  \Delta _P f\rvert  \rangle_P \lVert \Delta
_Q g\rVert_1 . 
\end{align}

Using the notation of \eqref{e:uv},  for integers $ v \geq r$, we have 
\begin{align*}
\sum_{P} \sum_{ \substack{ Q \;:\; Q\subset P \\    2 ^{v} \ell Q = \ell P}} 
\lvert \langle T (\Delta _P f \mathbf 1_{P \setminus P_Q}) , \Delta _Q g\rangle\rvert 
\lesssim 2 ^{- \eta ' v} B ^{0,v} (f,g)
\end{align*}
and by Lemma~\ref{l:uv}, this is in turn dominated by a choice of sparse form. Sparse forms are
again
dominated by a fixed form. We can sum this estimate over $ v \geq r$, so this case is complete. 

\subsubsection*{Second Subcase}
We turn attention to the second term in \eqref{e:insideSplit}, in which we have $ \langle \Delta _P
f \rangle _{P_Q} \mathbf 1_{S}$.   
This is the most intricate step, in that we combine several elementary steps.
The bound we prove is  uniform over a choice of $ S\in \mathcal S$.  Namely, 
\begin{equation} \label{e:InB}
 \begin{split}
  \Bigl\lvert 
 \sum_{ \substack{Q  \;:\;   Q^{\tau }  =S }}
\   \sum_{P  \;:\; Q \Subset P }  
 \langle T (\Delta _P f \cdot  \mathbf 1_{S}), \Delta _Q g \rangle 
 \Bigr\rvert  \lesssim  \langle \lvert  f\rvert \rangle_S \langle \lvert  g\rvert  \rangle_S \lvert 
S\rvert 
  \end{split} 
\end{equation}
This is the one point in the argument in which the implied constant depends upon the testing
constant $ \mathbf T$ in \eqref{e:T1}.

For each cube $Q$  with $ Q ^{\tau }=S$, define $\epsilon_Q$ by
\begin{equation}\label{e:eps}
\epsilon_Q \ave{|f|}_{S} : = \sum_{\substack{P \in \mathcal D,\ Q \Subset P_Q  }} 
\ave{\Delta_P f}_{P_Q}.
\end{equation}
By the first stopping condition, corresponding to the control of the averages of $ f$, $\{ \epsilon_Q
\}_{Q \in \mathcal D}$ is
uniformly bounded.   In particular, this operator is a martingale transform.  
\begin{equation*}
\Pi ^{\epsilon }_S g  = \sum_{Q \;:\; Q ^{\tau }=S}  \varepsilon _Q \Delta _Q g. 
\end{equation*}

We make the following observation about the second stopping condition, 
corresponding to the control of the averages of $ g$.  
Setting a conditional expectation on $ S$ to be 
\begin{equation*}
\mathbb E (\phi  \;\vert\; \mathcal F _S ) = 
\begin{cases}
\phi (x)    & x \in S \setminus F_s 
\\
\langle \phi  \rangle_ {S'}  &  x\in S',\  S'\in \mathcal F_S 
\end{cases}
\end{equation*}
Then, $ \lVert  \mathbb E (g \mathbf 1_{S} \;\vert \; \mathcal F_S ) \rVert _{\infty } \lesssim
\langle \lvert  g\rvert  \rangle_S $.  
We also have $ \Pi ^{\epsilon }_S g = \Pi ^{\epsilon }_S \mathbb E (g \mathbf 1_{S} \;\vert \;
\mathcal F_S ) $. 
Therefore, by the $ L ^2 $ bound for martingale transforms,  
\begin{equation}\label{e:Pi2}
\lVert \Pi ^{\epsilon }_S g\rVert_2 
\leq \lVert \mathbb E (g \mathbf 1_{S} \;\vert \; \mathcal F_S ) \rVert_2 \lesssim \langle \lvert 
g\rvert  \rangle_S \lvert  S\rvert ^{1/2} .  
\end{equation}

The point of our third stopping condition, corresponding to the control of the average of $ T
\mathbf 1_{S}$,  is that $  \mathbb E (T \mathbf 1_{S } \; \vert \;  \mathcal F_S)$ is bounded in $ L ^{\infty
}$ by a constant multiple of $ \mathbf T $.  Collecting these observations, we can rewrite our sum
as below, in which in the first step we use the definition \eqref{e:eps} to collapse the sum over $
P$. 
\begin{align}
 \textup{LHS of \eqref{e:InB}}
&= \lvert  \langle |f|  \rangle_S  \langle T \mathbf 1_{S}  , \Pi ^{\epsilon }_S g  \rangle \rvert 
 \\
 & =  
  \lvert  \langle |f|  \rangle_S  \langle T \mathbf 1_{S}  , \mathbb E (\Pi ^{\epsilon }_Sg\;\vert \;
\mathcal F_S)\rangle \rvert 
 \\
  & =  
  \lvert  \langle |f|  \rangle_S  \langle\mathbb E (T \mathbf 1_{S } \; \vert \;  \mathcal F_S) , \Pi
^{\epsilon }_S g\rangle \rvert 
  \\ 
 & \lesssim \langle |f|  \rangle_S   
 \lVert \mathbb E (T \mathbf 1_{S } \; \vert \;  \mathcal F_S) \rVert_2 
 \lVert  \Pi ^{\epsilon }_S
g\rVert_2 
 \label{e:InC}
 \lesssim \langle |f|  \rangle_S \langle |g| \rangle_S \lvert  S\rvert.  
\end{align}
This completes this case. 

\subsubsection*{Third Subcase} 
We address the top alternative in \eqref{e:insideSplit}, namely we bound 
\begin{equation}\label{e:3rd}
\sum_{S} 
\sum_{Q \;:\; Q ^{\tau } =S}  \sum_{ \substack{P \;:\; Q\Subset P  \\ P_Q  \subset S}} 
\langle \Delta _{P}f \rangle _{P_Q}\langle T \mathbf 1_{S \setminus P_Q} , \Delta _Q g \rangle
\end{equation}
This is similar to the first subcase, since $\mathbf 1_{S \setminus P_Q}$ is supported in $(2Q)^c$, then
the off-diagonal estimates also imply 
\begin{equation*}
|\inn{T \mathbf 1_{S \setminus P_Q} }{\Delta_Q g}| \lesssim P_{\eta}( \mathbf 1_{S \setminus P_Q} )(Q) \|
\Delta_Q g \|_1 \lesssim \left[ \frac{\ell Q}{\ell P} \right]^{\eta'} \| \Delta_Q g \|_1.
\end{equation*}
Holding the relative lengths of $ Q$ and $ P$ fixed, we then have for integers $ v \geq r$, 
\begin{equation*}
\sum_{S} 
\sum_{Q \;:\; Q ^{\tau } =S} \sum_{ \substack{P \;:\; Q\Subset P  \\ P_Q  \subset S, \ 2 ^{v} \ell Q
= \ell P}} 
|\inn{T (\Delta_P f \mathbf 1_{S \setminus P_Q}) }{\Delta_Q g}|
\lesssim 2 ^{- v \eta '} B ^{0,v} (f,g). 
\end{equation*}
We use the notation \eqref{e:uv}, and Lemma~\ref{l:uv} to complete this case. 

\subsubsection*{Fourth Subcase}
We address the bottom alternative in \eqref{e:insideSplit}, namely the case in which 
$ S=Q ^{\tau } \subsetneq P_Q $.  The point here is to gain geometric decay in the degree to which $
Q$ and $ P_Q$ are separated in the stopping tree $ \mathcal S$. 

Given $ S\in \mathcal S $, let $ S = S ^{(0)} \subsetneq  S ^{(1)} \subsetneq  \cdots \subsetneq 
P_0$
be the maximal chain of stopping cubes 
which contain $ S$, and continue up to $ P_0$.  For each  $ S_0 \in \mathcal S$, and integer $ t
\geq 1$, we bound 
\begin{equation}\label{e:4th}
\Bigl\lvert 
\sum_{S \;:\; S ^{(t)} = S_0}   
 \sum_{P \;:\;  S ^{(t-1)} \Subset P_Q   \subset S_0} 
\langle \Delta _{P}f \rangle _{P_Q}\langle T \mathbf 1_{P_Q \setminus S ^{(t-1)}} ,  \Pi_S  g \rangle
\Bigr\rvert
\lesssim 2 ^{- c t} \langle \lvert  f\rvert  \rangle _{S_0} \langle \lvert  g\rvert  \rangle _{S_0}
\lvert  S_0\rvert.  
\end{equation}

The point is to use the off-diagonal estimates, but there is a complication in that the stopping
cubes
are not good. 
To address this, we let $ \mathcal Q (S)$ be the maximal good cubes with $ Q ^{\tau } =S$, and set 
\begin{equation*}
\tilde \Pi _{Q ^{\ast} } g = \sum_{Q \;:\; Q ^{\tau }=S, Q\subset Q ^{\ast} } \Delta _Q g, \qquad Q
^{\ast} \in \mathcal Q (S). 
\end{equation*}
The goodness of the cubes implies that  $ \textup{dist} (Q  ^{\ast} , \partial S ^{(t-1)}) \geq
(\ell Q ^{\ast} )
^{\epsilon } (\ell S ^{(t-1)}) ^{1- \epsilon } 
\geq 2 ^{t/2} \ell Q ^{\ast} $, by \eqref{e:good}.    

The second point is that we have 
\begin{equation*}
\Bigl\lVert 
 \sum_{P \;:\;  S ^{(t-1)} \Subset P_Q   \subset S_0} 
\langle \Delta _{P}f \rangle _{P_Q}  \mathbf 1_{P_Q \setminus S ^{(t-1)}} 
\Bigr\rVert _{\infty } \lesssim \langle \lvert  f\rvert  \rangle_S 
. 
\end{equation*}
Combining these last two observations with \eqref{e:offdiaggood}, we see that for each $ Q ^{\ast}
\in \mathcal Q (S)$, 
\begin{align*}
\Bigl\lvert 
\sum_{S \;:\; S ^{(t)} = S_0}   
 \sum_{P \;:\;  S ^{(t-1)} \Subset P_Q   \subset S_0} 
\langle \Delta _{P}f \rangle _{P_Q}\langle T \mathbf 1_{P_Q \setminus S ^{(t-1)}} ,  \tilde \Pi_ {Q
^{\ast} }  g \rangle
\Bigr\rvert &\lesssim  2 ^{-t/2}\langle \lvert  f\rvert  \rangle_ {S_0}  \lVert \tilde \Pi_ {Q
^{\ast} }  g\rVert_1 
\\&
\lesssim  2 ^{-t/2}\langle \lvert  f\rvert  \rangle_ {S_0}   \langle \lvert  g\rvert  \rangle_S
\lvert  Q ^{\ast} \rvert.   
\end{align*}
Here we have used the stopping condition to dominate $\tilde \Pi_ {Q ^{\ast} }  g $.  
To conclude, we simply observe that 
\begin{align*}
\sum_{S \;:\; S ^{(t)} = S_0}\langle \lvert  g\rvert  \rangle_S   \sum_{Q ^{\ast} \in \mathcal Q
(S)} \lvert  Q ^{\ast} \rvert
& \leq \sum_{S \;:\; S ^{(t)} = S_0}\langle \lvert  g\rvert  \rangle_S   \lvert  S\rvert 
\lesssim \langle \lvert  g\rvert  \rangle _{S_0} \lvert  S_0\rvert.   
\end{align*}
Our proof of \eqref{e:4th} is complete. 

\subsection{The Near Terms} 
We address the term in \eqref{e:near}. Fix an integer $ v \geq r$, and consider $ Q\subset 3P
\setminus P$ with 
$ 2 ^{v} \ell Q = \ell P$. The cube $ Q$ is good, so that by  \eqref{e:good} and \eqref{e:ODE}, we
have 
\begin{equation*}
\lvert \langle  T \Delta _P f, \Delta _Q g \rangle\rvert  
\lesssim 2 ^{- v \eta '} \langle \lvert  \Delta _P f\rvert  \rangle_P \lVert \Delta _Q g\rVert_1 . 
\end{equation*}
But, then, we have 
\begin{equation*}
\lvert  \eqref{e:near}\rvert \lesssim 2 ^{- v \eta '} B ^{0,v} (f,g),  
\end{equation*}
where the latter bilinear form is defined in \eqref{e:uv}.  It follows from \eqref{e:uv} that the
near term is dominated by a sparse bilinear form.

\subsection{The Neighbors}
We bound the term in \eqref{e:neighbor}.  
For $ P$, let $ P', P''$ be  choices  children of $ P$. There are at most $ O(1)$ such choices. 
For integers $ 0\leq v \leq r$,  we bound 
\begin{equation}\label{e:Neigh}
\sum_{\substack{P\;:\;  P\subset P_0 }}  
\sum_{\substack{Q\;:\; \ell Q \leq \ell P = 2 ^{v } \ell Q,\ Q\cap 3P  \neq \emptyset
}} \langle T (\Delta _{P} f \cdot \mathbf 1_{P'})  ,  \mathbf 1_{P''}\Delta _{Q} g\rangle. 
\end{equation}

The case of $ P'\neq P''$ is straight forward. The function $\Delta _{P} f \cdot \mathbf 1_{P'} $ is
constant, so that the Hardy inequality immediately implies that 
\begin{align*}
\lvert  \langle T (\Delta _{P} f \cdot \mathbf 1_{P'})  ,  \mathbf 1_{P''}\Delta _{Q} g\rangle\rvert 
&\lesssim \lvert  \langle \Delta _P  \rangle _{P'}\rvert  \lvert  P'\rvert ^{1/2}   \lVert 
\mathbf 1_{P''} \Delta _Q g\rVert_2 
\\
& \lesssim \lvert  \langle \Delta _P  \rangle _{P'}\rvert  \cdot \lVert \Delta _Q g\rVert_1 . 
\end{align*}
And this can be summed to the bound we want.  Namely, it is dominated by $ B ^{0,v} (f,g)$, where
the last term is defined in \eqref{e:uv}.  

The case of $ P'=P''$ reduces to the testing inequality, and we have the same bound as above.

\subsection{The Far Term} 
We address the terms in \eqref{e:far}.  
For integers $ u, v\geq 1$,  we impose additional restrictions on $ P$ and $ Q$, and obtain a sparse
bound with geometric decay in these parameters.  From this, the required bound follows.  
Namely, we have for 
\begin{equation}\label{e:FAR}
\ell P= \ell P', \  P'\subset 3 ^{u-1} P, \quad    2 ^{v} \ell Q = \ell P,\  Q\subset 3 ^{u+1}P
\setminus  3 ^{u}P, 
\end{equation}
we have from \eqref{e:ODE} the estimate below. 
\begin{equation*}
\lvert  \langle T \Delta _{P'} f,  \Delta _Q g\rangle\rvert  \lesssim 2 ^{- \eta' (u+v)}  \langle
\lvert   \Delta _ {P'} f\rvert  \rangle_ {P'} \| \Delta _Q g \|_1 . 
\end{equation*}
Therefore, appealing to the definition in \eqref{e:uv}
\begin{equation*}
\sum_{P} \sum_{(P',Q) \textup{satisfy \eqref{e:FAR}}} 
\lvert  \langle T \Delta _{P'} f,  \Delta _Q g\rangle\rvert  
\lesssim  2 ^{- \eta' (u+v)}  B ^{u,v} (f,g). 
\end{equation*}
By Lemma~\ref{l:uv}, this case is complete.

\section{Lemmas} \label{s:lemma}

We collect three separate groups of Lemma,  (a) the sparse domination of a class of dyadic forms;
(b) standard off-diagonal estimates; and (c) a Hardy inequality.

\subsection*{Sparse Domination}
We define a class of (sub) bilinear forms that are basic to the proof.    For a cube $P$, let
$i_P = \log_2 (\ell P)$.   Let $ D _{k} f = \sum_{P \;:\; \ell P = 2 ^{k}} \Delta _P f$, and
define 
\begin{equation}\label{e:uv}
B ^{u,v} (f,g) 
= \sum_{P}  \langle  \lvert  D _{i_P-u} f\rvert  \rangle _{3P} 
\langle  \lvert  D _{i_P-v} g\rvert  \rangle _{3P} \lvert  P\rvert 
\end{equation}
Above, $ u, v\geq 0$ are fixed integers, so that we are taking the martingale differences that are
somewhat smaller, over the triple of $ P$.  We comment that this is a dyadic operator of complexity
$ u+v$, in the language of \cite{MR2912709}. 

We remark that a standard argument would  write 
\begin{align} \label{e:Bint}
B ^{u,v} (f,g) 
& =  \int \sum_{P}  \langle  \lvert  D _{i_P-u} f\rvert  \rangle _{3P} 
\langle  \lvert  D _{i_P-v} g\rvert  \rangle _{3P}  \mathbf 1_{P}(x) \; dx 
\end{align}
It is clear that we would dominate this last integral by a product of square functions  
$ \int S_ u f \cdot S_v g \; dx$, with the square functions  defined by 
\begin{equation}  \label{e:S}
(S _{u} f) ^2 = 
\sum_{P}  \langle  \lvert  D _{i_P-u} f\rvert  \rangle _{3P}  ^2 \mathbf 1_{P}. 
\end{equation}
The deepest fact needed in our proof of the $ T1$ theorem is this:  The square functions $ S _{u} $
are weakly bounded, with constant linear in $ u$.  

\begin{lemma}\label{l:Su} We have the inequality below, valid for all integers $ u\geq 0$
\begin{equation}\label{e:Su<}
\lVert S_u f  \;:\;  L ^{1} \mapsto L^{1, \infty }\rVert \lesssim (1+ u).  
\end{equation}
\end{lemma}

\begin{proof}
The square function $ S_u$ is bounded on $ L ^2 $, with constant independent of $ u$, by the
orthogonality of martingale differences.   To prove the weak-type inequality, we take $ f\in L
^{1}$, and apply the Calder\'on-Zygmund decomposition at height 1.  Thus, $ f = g+b$, where $ \lVert
g\rVert_2 \lesssim  \lVert f\rVert_1 ^{1/2}  $, and we have 
\begin{equation*}
b = \sum_{B\in \mathcal B} b_ B , 
\end{equation*}
where $ \mathcal B$ consists of disjoint dyadic cubes with $ \sum_{B\in \mathcal B} \lvert  B\rvert
\lesssim \lVert f\rVert_1 $, and $ b_B$ is supported on $ B$, has integral zero, and $ \lVert
b_B\rVert_1 \lesssim | B |$.  

We do not estimate $ S_u f $ on the set $ E = \bigcup _{B\in \mathcal B} 3 B$.  And estimate 
\begin{align*}
\lvert  \{ x\not\in E \;:\;   S_u f (x) > 2\}\rvert  
& \leq \lvert  \{  S_u g > 1\}\rvert +  \lvert  \{ x\not\in E \;:\;   S_u b (x) > 1\}\rvert  
\end{align*}
The first term is controlled by the $ L ^2 $ bound and the fact that $ \lVert g\rVert_2 ^2 \leq
\lVert f\rVert_1$. 

Concerning the function $ b$, observe that for $ P\not\subset E$, that we have  $ \langle  \lvert  D
_{i_P-u} f\rvert  \rangle _{3P} \neq 0$ only if there is some $ B \in \mathcal B$ with $ B\subset
3P$, and $ 2 ^{u}\ell B \geq \ell P$. 
For a fixed $ B$, there are only $ 3 ^{d} (1+u) $ such choices of $ P$.  
Therefore, we will estimate
\begin{align*}
 \lvert  \{ x\not\in E \;:\;   S_u b (x) > 1\} \rvert  
 &\lesssim  \sum_{P \;:\; P\not\subset E} \int _{P} \lvert  \Delta b \rvert\;dx 
 \\
 & \lesssim  
 \sum_{v=1} ^{u} 
 \sum_{P \;:\; P\not\subset E} \sum_{\substack{B\in \mathcal B \;:\; B\subset P\\ 2 ^{v} \ell B =
\ell P}}
 \int _{P} \lvert  \Delta b_B \rvert\;dx
 \\
 & \lesssim 
  \sum_{v=1} ^{u} 
 \sum_{P \;:\; P\not\subset E} \sum_{\substack{B\in \mathcal B \;:\; B\subset P\\ 2 ^{v} \ell B =
\ell P}} \lvert  B\rvert 
  \lesssim u \sum_{B\in \mathcal B} \lvert  B\rvert \lesssim u \lVert f\rVert_1.   
\end{align*}
Our proof is complete. 
\end{proof}

The previous estimate is the principal tool in this sparse bound, which we use repeatedly in our
proof of the sparse result. 

\begin{lemma}\label{l:uv} For all $ u,v \geq 0$, all bounded compactly supported functions $ f, g$,
there is a sparse collection $ \mathcal S$ so that 
\begin{equation*}
B ^{u,v} (f,g) \lesssim (1+u)(1+v) \Lambda (f,g).  
\end{equation*}
\end{lemma}

It is an easy corollary from the conclusion above for $ u,v=0$ that martingale transforms satisfy a
sparse bound.  And, we also comment that the linear dependence of the constant above presents no
difficulty in application, as we will always have a term that decreases geometrically in $ u +v$.  

\begin{proof}
Note that from the equality for $ B ^{u,v}$ in \eqref{e:Bint}, we have 
\begin{align*}
B ^{u,v} (f,g) \lesssim \int S_u f \cdot S_v g \; dx 
\end{align*}
with the square functions  defined by  \eqref{e:S}.  
But, we localize this familiar argument.  Define 
\begin{equation*}
(S _{u, P_0} f)^2 = \sum_{P \;:\; P \subset P_0}  \langle  \lvert  D _{i_P-u} f\rvert  \rangle
_{3P} ^2 \mathbf 1_{P}, 
\end{equation*}
we have for an absolute constant $ C$, and all choices of $ u \geq 0$, 
\begin{equation}\label{e:sq}
\lvert  \{x \in 3 P_0 \;:\;  S _{u,P_0} f > C (1+u) \langle \lvert  f\rvert  \rangle _{3
P_0}\}\rvert \leq
\tfrac 18 \lvert   P_0\rvert.    
\end{equation}
Moreover, the set on the left is contained in $ P_0$.

We construct the sparse bound this way.  Fix a large (non-dyadic) cube $ P_0$ that  $ \tfrac 12 P_0$
contains the  support of both $ f$ and $ g$. The sparse cubes outside of $ P_0 $ can be taken to $ 3
^{k} P_0$, for $ k\in \mathbb N $. 
We need to construct the sparse collection inside of $ P_0$.  
Consider the  restricted sum 
\begin{equation}\label{e:p0}
I (P_0) :=  \int \sum_{P \;:\; P\subset P_0}  \langle  \lvert  D
_{i_P-u} f\rvert  \rangle _{3P} 
\langle  \lvert  D _{i_P-v} g\rvert  \rangle _{3P}  \mathbf 1_{P} \; dx  . 
\end{equation}
Using \eqref{e:sq}, set 
\begin{equation*}
E _0  = \{ S _{u, P_0} f > C (1+u) \langle \lvert  f\rvert  \rangle _{3P_0}\} 
\cup \{ S _{v, P_0} g > C  (1+v)\langle \lvert  g\rvert  \rangle _{3P_0}\} . 
\end{equation*}
This set is contained in $ P_0$, and has measure at most $ \tfrac 14 \lvert  P_0\rvert $.  
Let $ \mathcal E_0$ be the maximal dyadic components of $ E_0$.  We have by Cauchy-Schwartz and
construction, 
 \begin{align*}
I (P_0) & \leq C ^2 \langle \lvert  f\rvert  \rangle _{3P_0} \langle \lvert  g\rvert  \rangle
_{3P_0}  \lvert  P_0\rvert 
+ \sum_{ Q\in \mathcal E_0} I (Q).  
\end{align*}
The first term on the right is the first term in our sparse bound.  We recurse on the second terms.
This completes the proof.
\end{proof}

A very general fact about sparse forms is that they admit a `universal domination.'   
\begin{lemma}\label{l:sparse<} 
Given $ f, g$, there is a sparse operator $ \Lambda _0$, and constant $ C>1$ so that for any other
sparse operator $ \Lambda $, we have $ \Lambda (f,g) < C \Lambda_0 (f,g)$.  

\end{lemma}

\begin{proof}
Recall that shifted dyadic grids are a collection $ \mathbf G$ of at most $ 3 ^{d}$ dyadic grids $
\mathcal G \in \mathbf G$, 
so that \emph{every} cube $ Q\subset \mathbb R ^{d}$ can be approximated by some cube in a dyadic
grid $ \mathcal G\in \mathbf G$.  Namely, for each cube $ Q$, there is a $ \mathcal G$ and a cube $
P\in \mathcal G$ so that 
$ \tfrac 16 \ell (P) \leq \ell (Q ) $ and $ Q\subset 6 P$.   See \cite{MR3065022}*{Lemma 2.5} for an
explicit proof.

Shifted grids permit us to construct a universal sparse operator for each grid $ \mathcal G\in
\mathbf G$. 
We show this: For any dyadic grid $ \mathcal G$, let $ \mathcal S \subset \mathcal G $ be 
such that for $ S\in \mathcal S$, there is a set $ E_S  \subset S$ so that $ \lvert  E_S\rvert > c
\lvert  S\rvert  $ and 
$ \lVert \sum_{S\in \mathcal S} \mathbf 1_{E_S} \rVert _{\infty } \leq c ^{-1} $.  
Given non-negative $ f, g$ bounded and compactly supported, we construct $ \mathcal U _{\mathcal G}
\subset \mathcal G$ so that  there are pairwise disjoint exceptional sets $ \{E _{Q} \;:\; Q\in
\mathcal U _{\mathcal G}\}$ so that $ E_Q\subset Q$ and $ \lvert  E_Q\rvert \geq \tfrac 12 \lvert 
Q\rvert  $, and moreover, 
\begin{equation}\label{e:su}
\sum_{S\in \mathcal S} \langle f \rangle_S \langle g \rangle_S \mathbf 1_{S} 
 \leq 16 ^{ d} c ^{-2}  \sum_{U \in \mathcal U _{\mathcal G}}  \langle f \rangle_U \langle g
\rangle_U
\mathbf 1_{U}.  
\end{equation}
To complete the proof of the Lemma, we remark that the collection $ \{ \mathcal U _{\mathcal G}
\;:\; \mathcal G\in \mathbf G\}$ is sparse. 
It dominates every sparse operator formed from some $ \mathcal G\in \mathbf G$, hence is universal
for all sparse operators. 

\smallskip 
For integers $ k$, let $ \mathcal U _{k}$ be the maximal cubes $ Q\in \mathcal G$ so that $ \langle
f \rangle_Q \langle g \rangle_Q \geq 8 ^{2 dk}$. Then, the product is at most   $ 8 ^{2d k+ 2d/3}$.  The
cubes $ Q\in \mathcal U_k$ are pairwise disjoint, by maximality.  
We check that the children are small in measure.  Setting $ \mathcal C (Q) =\{P\in \mathcal U _{k+1}
\;:\; P\subsetneq Q\}$, 
we can  write  $  \mathcal C (Q) = \mathcal C _{f}(Q) \cup \mathcal C _{g}(Q)$, where $ P\in \mathcal C _{f}(Q)$
if $ P\in \mathcal C(Q)$ and $ 
\langle f \rangle_P > 4 ^{d} \langle f \rangle_Q$, and similarly for $ \mathcal C _{g}(Q)$. But, then
it is clear that 
\begin{equation*}
\sum_{P\in \mathcal C _{f}(Q)} \lvert  P\rvert \leq 4 ^{-d} \lvert  Q\rvert  \leq \tfrac 14 \lvert  Q\rvert. 
\end{equation*}
We set $ E_Q = Q \setminus \bigcup _{P\in \mathcal C(Q)} P$. This set has measure at least $ \tfrac{1}{ 2 }
\lvert  Q\rvert $.  

Set $ \mathcal U _{\mathcal G} = \bigcup _{k} \mathcal U_k$.  The sets $ \{E_Q \;:\; Q\in \mathcal
U\}$ are pairwise disjoint. 
Now, given the sparse collection as above, provided $ \langle f \rangle_S \langle g \rangle_S \neq
0$,  each $ S\in \mathcal S$ has a parent $ S ^{u} \in \mathcal U$, namely the smallest element of $
\mathcal U$ that contains $ S$.  
Then, 
\begin{align*}
\sum_{S\in \mathcal S} \langle f \rangle_S \langle g \rangle_S \mathbf 1_{S} 
& = 
\sum_{U \in \mathcal U _{\mathcal G}}  
\sum_{\substack{S\in \mathcal S\\ S ^{u} = U }} \langle f \rangle_S \langle g \rangle_S \mathbf
1_{S}  
\\
& \leq 16 ^{d} 
\sum_{U \in \mathcal U _{\mathcal G}}  \langle f \rangle_U \langle g \rangle_U 
\sum_{\substack{S\in \mathcal S\\ S ^{u} = U }}   \mathbf 1_{S}  
 \leq 16 ^{d} c ^{-2}  \sum_{U \in \mathcal U _{\mathcal G}}  \langle f \rangle_U \langle g
\rangle_U
\mathbf 1_{U}.  
\end{align*}
This verifies \eqref{e:su}, so completes the proof. 

\end{proof}

\subsection*{Off-Diagonal Estimates}
We begin with the very common off-diagonal estimate. For $\eta > 0$ consider the Poisson-like
operator
\begin{equation*}\label{Poisson}
P_{\eta} \Phi (Q) : = \int_{\mathbb R^d} \frac{(\ell Q)^{\eta} \Phi(y)}{(\ell Q)^{d+ \eta} + \textup{dist}(y,
Q)^{d+ \eta}} \; dy.
\end{equation*}

\begin{lemma}[Off-diagonal estimate]\label{offdiagestimate}  Let $g$ be a function with $\int g \,
dx
= 0$, supported on a
cube $Q$, and $f \in L^2$ supported on $(2Q)^c$, then we have
\begin{equation} \label{e:ODE}
|\inn{Tf}{g}| \lesssim P_{\eta}|f|(Q) \| g \|_1 \leq P_{\eta}|f|(Q) |Q|^{1/2} \| g \|_2.
\end{equation}
\end{lemma}

{\bf Proof: } Let $x_Q$ be the center of $Q$, then we have
\begin{align*}
|\inn{Tf}{g}| & = \left| \int_Q \int_{(2Q)^c}  K(x,y)f(y)g(x) \, dy \, dx \right| = \left|
\int_{(2Q)^c} \int_Q  (K(x,y) - K(x_Q,y) )f(y)g(x) \, dx \, dy \right| \\
& \leq \mathcal K_T \int_{(2Q)^c} \int_Q \frac{|x-x_Q|^{\eta}}{|x-y|^{d+\eta}} |f(x)g(y)| \, dx \, dy 
\lesssim \mathcal K_T P_{\eta}|f|(Q) \| g \|_1.
\end{align*}
And the second inequality follows from Cauchy-Schwarz. \qed

\begin{lemma}\label{offdiaggood}
Suppose that $Q \Subset P$ and $Q$ is good, then there is $\eta' = \eta'(\eta, \gamma) > 0$, such
that
\begin{equation}\label{e:offdiaggood}
P_{\eta} \mathbf 1_{\mathbb R^d \setminus P}(Q) \lesssim \left[ \frac{\ell Q}{\ell P} \right]^{\eta'}. 
\end{equation}\end{lemma}

\begin{proof}
Let $\lambda = (\ell P / \ell Q)^{1-\gamma}$. By goodness of $Q$, Lemma \ref{GeomGood} implies
\begin{align*}
P_{\eta} \mathbf 1_{\mathbb R^d \setminus P}(Q) & = \int_{\mathbb R^d \setminus P} \frac{(\ell Q)^{\eta}}{(\ell Q)^{d+ \eta} +
\textup{dist}(y,Q)^{d + \eta}} \; dy & \\
& \leq \int_{\mathbb R^d} \frac{(\ell Q)^{\eta}}{((\ell Q)^{\gamma} (\ell P)^{1- \gamma})^{d+ \eta} +
\textup{dist}(y,Q)^{d + \eta}} \; dy \\
& \leq \left[ \frac{\ell Q }{\ell P} \right]^{\eta(1- \gamma)}  P_{\eta} \mathbf 1_{\mathbb R^d }(\lambda Q).
\end{align*}
So, the result follows.
\end{proof}

\subsection*{Hardy's Inequality}

This is the version of Hardy's inequality that we need. It can be proved from the one dimensional
version.  
In point of fact, we only need this in the case where the function $ f$ below is constant. 
\begin{lemma}\label{l:hardy}  For any cube,  $ P$, and $ 1 < p < \infty $, we have 
\begin{equation}\label{e:hardy}
\int _{3P \setminus P} \int _{P}  \frac {f (x) g (y)} { \lvert  x-y\rvert ^{n} } \; dxdy 
\lesssim \lVert f\rVert_p \lVert g\rVert_{p'}.  
\end{equation}

\end{lemma}

\bibliographystyle{amsplain}


\end{document}